\documentclass[12pt]{amsart}
\usepackage{geometry} 
\usepackage{pgfplots}
\usepackage{psfrag}
\usepackage{amsmath}
\usepackage{calc}
\usepackage{systeme}
\usepackage{amsfonts}
\usepackage{amsthm}
\usepackage {amssymb}
\usepackage{algorithm2e}
\usepackage[applemac]{inputenc}
\usepackage{bbold}
\usepackage[numbers]{natbib}
\usepackage{graphicx}
\usepackage{caption}
\usepackage{subcaption}

\newtheorem{Theorem}{Theorem}
\newtheorem{Lemma}{Lemma}
\newtheorem{Rem}{Remark}

\newtheorem{Prop}{Proposition}

\usepackage[foot]{amsaddr}

\title{A remark on Uzawa's algorithm and an application to mean field games systems}
\author{Charles Bertucci}

\address{Université Paris-Dauphine, PSL Research University,UMR 7534, CEREMADE, 75016 Paris, France}
\date{} 

\begin{document}
\begin{abstract}
In this paper, we present an extension of Uzawa's algorithm and apply it to build approximating sequences of mean field games systems. We prove that Uzawa's iterations can be used in a more general situation than the one in it is usually used. We then present some numerical results of those iterations on discrete mean field games systems of optimal stopping, impulse control and continuous control.
\end{abstract}
\maketitle
\tableofcontents
\section*{Introduction}
This paper is concerned with the study of an extension of Uzawa's algorithm. We show that the standard Uzawa's algorithm can be used to find solutions of systems similar to the ones characterizing saddle points of lagrangians, even though there is not a proper langrangian associated with this system. The second part of this paper is concerned with the application of this remark to build approximating sequences of solutions of Mean Field Games (MFG) systems.\\

Uzawa's algorithm was introduced to solve minimization problems with constraints. The main idea of this algorithm is to use a projected gradient descent on the dual problem. Because of its simplicity and efficiency, Uzawa's algorithm is often used in practical problems. We recall that the output of this algorithm is a sequence which converges toward the solution of the primal minimization problem. In the first part of this paper we prove that we can use the same algorithm to find solutions of a wider class of systems than the ones which characterize saddle points of lagrangians.\\

Next, we take full advantage of this remark to build approximating sequences for several MFG systems. MFG have been introduced by J.-M. Lasry and P.-L. Lions in \citep{lasry2007mean} and independently and in a particular case by M. Huang, P. Caines and R. Malhamme in \citep{huang2006large}. The theory of MFG is concerned with Nash equilibria of differential games with infinitely mainly indiscernable players, who interacts only through mean field type terms. We refer to \citep{lions2007cours} for a detailed presentation of MFG and to \citep{carmona2017probabilistic} for a complete presentation of the probabilistic theory of MFG. In general, the study of a MFG requires to solve the so-called master equation, see \citep{lions2007cours,cardaliaguet2015master}, but in the case when there is no common noise, the problem reduces to a system of Partial Differential Equations (PDE). It is well known that in the so-called potential case, MFG systems can be interpreted as the optimality conditions for an optimal control problem of a PDE, see \citep{lasry2007mean,cardaliaguet2010notes} for instance. Thus Uzawa's algorithm is a natural method we can apply to such optimal control problems. We show that, under monotonicity assumptions, we can apply an Uzawa's like algorithm to approximate solutions of MFG systems even in the non-potential case. In this paper we shall apply this algorithm to MFG systems of continuous control (i.e. as in \citep{lasry2007mean}), optimal stopping (see \citep{bertucci2017optimal}) and impulse control (see \citep{bertucci2018fokker}).\\

The last part of this paper presents the results of the implementation of Uzawa's iterations to the discretized problems of MFG of optimal stopping, impulse control and continuous control. 

\subsection*{Bibliographical comments}
We here give some details on the bibliographical context in which this article takes place. Concerning the literature regarding Uzawa's algorithm, there exist plenty of results on this well known algorithm. Although, using this algorithm to find solutions of systems of inequalities seems to be used only in the case of linear system, as in \citep{bramble1997analysis,elman1994inexact} for instance.\\

Concerning the MFG literature, the first numerical methods for MFG systems have been developed by Y. Achdou and I. Capuzzo-Dolcetta in \citep{achdou2010mean}. Several other methods have been studied and some of them involved the optimal control interpretation in the potential case. Such methods are somehow similar to the one we present here because they are also the implementation of a search for saddle points. We refer to \citep{bricen?o2018proximal,benamou2015augmented} for example. The main novelties of our work is to consider the non potential case and that we consider the cases of optimal stopping and impulse control. Furthermore, we mention the papers \citep{ferreira2015existence,almulla2017two} of R. Ferreira, D. Gomes and al. in which the first order MFG system of continuous control is interpreted as a system of variational inequalities and solve numerically. The interpretation in terms of variational inequalities of the MFG system is central in the rest of this paper.

\section{A remark on Uzawa's algorithm}
\subsection{Presentation of the standard algorithm}
We present here the classical result of convergence of Uzawa's algorithm. Although we are going to present this algorithm from the point of view of the search of a saddle point, let us recall the well-known fact that given a convex minimization problem, Uzawa's algorithm is only the projected gradient ascent method applied on the dual problem. Let us take a lagrangian $L$ defined by:
\begin{equation}\label{defL}
L(x,y) = F(x) + < a(x),b(y)> , \forall x \in K_1, \forall y \in K_2;
\end{equation}
where $K_1$ is a closed convex subset of the Hilbert space $(H_1, (\cdot,\cdot))$ and $K_2$ is a closed convex subset of the Hilbert space $\big(H_2, ((\cdot, \cdot)) \big)$. We denote by $(H_3, < \cdot, \cdot>)$ a third Hilbert space. The applications $a : H_1 \to H_3$ and $b: H_2 \to H_3$ are such that for all $y \in K_2$, $x \to < a(x), b(y)>$ is a convex application and $\tilde{K}_2 := b(K_2)$ is closed and convex. Moreover $F : H_1 \to \mathbb{R}$ is a convex function. The lagrangian $L$ is associated to the following minimization problem :
\[
\underset{x \in K_1}{\inf} \{F(x) + \underset{y \in K_2}{\sup} <a(x), b(y)>\}.
\]
We recall that a saddle point of $L$ is a couple $(x,y) \in K_1 \times K_2$ such that
\begin{equation}\label{saddlepoint}
\underset{x' \in K_1 y' \in K_2}{\inf \sup} L(x',y') = \underset{y' \in K_2 x' \in K_1}{\sup \inf} L(x',y') = L(x,y).
\end{equation}
We fix a real number $\delta > 0$ and we denote by $\bold{P}_A$ the orthogonal projection on the set $A$ in $H_3$. Uzawa's algorithm (with step $\delta$) consists in building the sequence $(x_n, y_n)_{n \in \mathbb{N}}$ as follows :
\begin{equation}\label{iteration1}
\begin{cases}
y_0 \in K_2;\\
x_n = \underset{x \in K_1}{\text{ arginf }} \{F(x) + < a(x), b(y_n)>\};\\
y_{n+1} \in b^{-1} \big(\{ \bold{P}_{\tilde{K_2}}(b(y_n) + \delta a(x_n))\} \big);
\end{cases}
\end{equation}
where we recall that $\tilde{K_2} = b(K_2)$. Before presenting a convergence result for those iterations, we introduce the following definition. An application $f$ from the Hilbert space $(H, <\cdot, \cdot>)$ into itself is said to be $\alpha$ monotone if for any $x,y \in H$, 
\[
<f(x) - f(y), x - y> \geq \alpha <x-y, x-y>.
\]
An application $0$ monotone is simply called monotone and an application $f$ is said to be strictly monotone if for any $x,y \in H$ such that $x \ne y$ the following holds 
\[
<f(x) - f(y), x - y> > 0.
\]
A classical convergence result concerning the sequence $(x_n, y_n)_{n \in \mathbb{N}}$ is the following :
\begin{Theorem}\label{classicaluzawa}
Let us assume that :
\begin{itemize}
\item The application $F$ is differentiable with differential $f$ which is $\alpha $ monotone.
\item The application $a$ is $C$ lipschitz for some constant $C >0$.
\end{itemize}
Then if $ \delta < \frac{ 2 \alpha}{C^2}$, for any $y_0 \in K_2$, the sequence $(x_n,y_n)_{n \in \mathbb{N}}$ defined by (\ref{iteration1}) is well defined and $(x_n)_{n \in \mathbb{N}}$ converges toward $x^*$ in $H_1$, where $(x^*, y^*)$ is the unique saddle point of $L$.
\end{Theorem}
We briefly recall here the proof of this result.
\begin{proof}
Given the assumptions we made, we know that there exists a unique couple $(x^*, y^*) \in H_1 \times H_2$ verifying (\ref{saddlepoint}). It satisfies 
\[
\begin{cases}
(f(x^*), x' -x^*) + < a(x') - a(x^*), b(y^*)> \geq 0, \forall x' \in K_1;\\
< a(x^*), b(y') - b(y^*)> \leq 0, \forall y' \in K_2.
\end{cases}
\]
Moreover by construction, $(x_n)_{n \in \mathbb{N}}$ satisfies for all $n \in \mathbb{N}$ :
\[
(f(x_n), x' -x_n) + < a(x') - a(x_n), b(y_n)> \geq 0, \forall x' \in K_1.
\]
Thus we deduce that 
\[
\begin{aligned}
<a(x_n) - a(x^*), b(y^*) - b(y_n)> &\geq (f(x_n) - f(x^*), x_n - x^*) \\
& \geq \alpha ||x_n - x^*||_{H_1}.
\end{aligned}
\]
Because $\bold{P}_{\tilde{K_2}}$ is a contraction, we obtain that
\[
\begin{aligned}
||b(y_{n+1}) - b(y^*)||_{H^3}^2  \leq &||b(y_n) - b(y^*) + \delta (a(x_n) - a(x^*))||_{H^3}^2\\
 \leq &||b(y_n) - b(y^*)||_{H^3}^2 + 2 \delta < b(y_n) - b(y^*), a(x_n) - a(x^*)> \\
 &+ \delta^2 ||a(x_n) - a(x^*)||_{H^3}^2.
\end{aligned}
\]
We then deduce that 
\[
||b(y_{n+1}) - b(y)||_{H^3}^2  \leq ||b(y_n) - b(y)||_{H^3}^2  - 2\delta\alpha ||x_n - x ||_{H^1}^2 + \delta^2 ||a(x_n) - a(x)||_{H^3}^2.
\]
Finally, because $\delta < \frac{2 \alpha}{C^2}$ we obtain that there exists $0 < \beta <1$ such that :
\[
\beta ||x_n - x ||_{H^1}^2 \leq ||b(y_n) - b(y)||_{H^3}^2 - ||b(y_{n+1}) - b(y)||_{H^3}^2,
\]
which concludes the proof of the result.
\end{proof}
\begin{Rem}
The use of the application $b$ and of the Hilbert space $H_2$ is somewhat artificial. We only use this formalism because it is closer to the set up needed for the applications of the next section. Moreover let us note that we do not state any convergence for the sequence $(y_n)_{n \geq 0}$.
\end{Rem}
\subsection{A generalization of Uzawa's algorithm}
We now remark that instead of using (\ref{iteration1}) to define a sequence $(x_n, y_n)_{n \in \mathbb{N}}$, we can use the following :
\begin{equation}\label{iteration2}
\begin{cases}
y_0 \in K_2;\\
x_n \text{ is defined by }  (f(x_n), x' -x_n) + < a(x') - a(x_n), b(y_n)> \geq 0, \forall x' \in K_1;\\
y_{n+1} \in b^{-1} \big(\{ \bold{P}_{\tilde{K_2}}(b(y_n) + \delta a(x_n))\} \big).
\end{cases}
\end{equation}
Let us note that if $F$ is a convex differentiable function, then (\ref{iteration1}) and (\ref{iteration2}) are equivalent, but the second one is more general in the sense that it allows us to build the sequence $(x_n, y_n)_{n \in \mathbb{N}}$ even in the case in which there is no function $F$ for which $f$ is the differential. Under the assumptions of theorem \ref{classicaluzawa}, the sequence $(x_n)_{n \in \mathbb{N}}$ converges toward $x^*$ where $(x^*,y^*)$ is the saddle point of $L$. Here we are interested in approximating the couples $(x_*, y_*) \in H_1 \times H_2$ solutions of
\begin{equation}\label{systemvi}
\begin{cases}
(f(x_*), x' -x_*) + < a(x') - a(x_*), b(y_*)> \geq 0, \forall x' \in K_1;\\
< a(x_*), b(y') - b(y_*)> \leq 0, \forall y' \in K_2;\\
x \in K_1 ; y \in K_2.
\end{cases}
\end{equation}
We establish the following result :
\begin{Theorem}\label{newiteration}
Let us take $f : H_1 \to H_1$. We assume that :
\begin{itemize}
\item The application $f$ is $\alpha$ monotone.
\item The application $a$ is $C$ lipschitz for some constant $C>0$ and differentiable.
\item There exists a couple $(x_*,y_*)$ satisfying (\ref{systemvi}).
\end{itemize}
Then, for any $y_0 \in K_2$, if $\delta < \frac{2\alpha}{C^2}$, (\ref{iteration2}) defines indeed a sequence $(x_n, y_n)_{n \in \mathbb{N}}$ and $(x_n)_{n \in \mathbb{N}}$ converges toward $x_*$ in $H_1$.
\end{Theorem}
\begin{proof}
First let us remark that for any $y_0 \in K_2$, the sequence $(x_n, y_n)_{n \in \mathbb{N}}$ is well defined. Indeed the second line of (\ref{iteration2}) defines a unique element $x_n \in H_1$ for any $n \in \mathbb{N}$. This comes from the fact that for any $y \in K_2$, $\epsilon > 0$, we can define the sequence $(\xi_p)_{p \in \mathbb{N}}$ by :
\[
\begin{cases}
\xi_0 \in H_1;
\xi_{p+1} = \bold{P}_{K_1}\big(\xi_p - \epsilon \big( f(\xi_p) + (Da(\xi_p))^*(b(y)) \big) \big);
\end{cases}
\]
where $\bold{P}_{K_1}$ stands for the orthogonal projection (for $H_1$) onto $K_1$. If $\epsilon$ is small enough, the sequence $(\xi_p)_{p \in \mathbb{N}}$ is a Cauchy sequence whose limit $\xi_*$ satisfies 
\[
(f(\xi_*), x' -\xi_*) + < a(x') - a(\xi_*), b(y)> \geq 0, \forall x' \in K_1.
\]
Such an element is unique because of the $\alpha$ monotonicity of $f$.\\

The rest of the proof follows the same argument as in the proof of theorem \ref{classicaluzawa}.
\end{proof}
\begin{Rem}
The existence of a couple $(x_*,y_*)$ satisfying (\ref{systemvi}) can be obtained directly under some assumptions on the applications $a$ and $b$ via a Kakutani's type fixed point theorem.
\end{Rem}

\section{Application of Uzawa's iterations to mean field games}
We now present how we can use the previous results to approximate some MFG systems. We shall apply this remark on Uzawa's algorithm to three different MFG systems. The first one is a system modeling a MFG of optimal stopping as in \citep{bertucci2017optimal}. The second one is a MFG system modeling an impulse control problem, following \citep{bertucci2018fokker} and we finish with the classical case of continuous control as in \citep{lasry2007mean}. To simplify notations, we present the following results in a stationary setting in which the state space is the $d$ dimensional torus $\mathbb{T}^d$.

\subsection{The case of optimal stopping}
We are here interested in approximating the solution of the following system of unknown $(u,m)$ :
\begin{equation}\label{optstopvi}
\begin{cases}
\forall v \in H^1(\mathbb{T}^d), v \leq 0 :\\
(f(m) + \nu \Delta u - \lambda u, v -u)_{H^{-1}\times H^1} \geq 0;\\
(-\nu \Delta m + \lambda m - \rho, v - u)_{H^{-1}\times H^1} \leq 0 ;\\
\int_{\mathbb{T}^d}(f(m) + \nu \Delta u - \lambda u)m = 0;\\
u \leq 0 ; m \geq 0;
\end{cases}
\end{equation}
where $f$ is a continuous application from $L^2(\mathbb{T}^d)$ into itself, $\nu, \lambda >0$ are two parameters of the model and $\rho \in H^{-1}(\mathbb{T}^d)$ is the entry rate of the players. The exit cost of the MFG is here $0$. The first variational inequality of this system arises from the obstacle problem satisfied by the value function $u$ of a generic player. The second variational inequality and the integral relation arise from the "Fokker-Planck equation" satisfied by the density of players $m$. Let us remark that we have here abused the name variational inequality as we only refer to a variational formulation which is an inequality and not to the famous concept introduced in \citep{lions1967variational} by Lions and Stampacchia. This system models Nash equilibria in mixed strategies of a MFG of optimal stopping, we refer to \citep{bertucci2017optimal} for more details on this system. From \citep{bertucci2017optimal} we know that there exists a unique solution $(u,m) \in H^2(\mathbb{T}^d) \times H^1(\mathbb{T}^d)$ of (\ref{optstopvi}) under the assumption that $f$ is strictly monotone, i. e. that it satisfies for all $m,m' \in L^2(\mathbb{T}^d)$:
\[
\int_{\mathbb{T}^d}(f(m) - f(m'))(m- m') > 0 \text{ if } m \ne m'.
\]
Let us remark that $(u,m)$ also satisfies 
\begin{equation}\label{stopgoodvi}
\begin{cases}
 \forall µ \in L^2(\mathbb{T}^d), µ \geq 0:\\
\int_{\mathbb{T}^d} (f(m) + \nu \Delta u - \lambda u) (µ - m) \geq 0;\\
\forall v \in H^1(\mathbb{T}^d), v \leq 0 :\\
(-\nu \Delta m + \lambda m - \rho, v - u)_{H^{-1} \times H^1} \geq 0.
\end{cases}
\end{equation}
In the case when $f$ is strictly monotone, (\ref{stopgoodvi}) has a unique solution $(u,m) \in H^2(\mathbb{T}^d) \times H^1(\mathbb{T}^d)$ which is the unique solution of (\ref{optstopvi}). The system (\ref{stopgoodvi}) falls under the scope of application of the previous section. Thus we define for $\delta > 0$ the following Uzawa's iterations :
\begin{equation}\label{iterationstop}
\begin{cases}
u_0 \in H^2(\mathbb{T}^d), u_0 \leq 0.\\
m_n \in L^2(\mathbb{T}^d) \text{ defined by : } \\
m_n \geq 0; \forall µ \in L^2(\mathbb{T}^d), µ \geq 0 : \int_{\mathbb{T}^d}(f(m_n) + \nu \Delta u_n - \lambda u_n)(µ - m_n) \geq 0.\\
u_{n+1} \in H^2(\mathbb{T}^d) \text{ defined by : } L u_{n+1}  = \bold{P}_{\tilde{K}} \big( L u_{n} - \delta(m_n - L^{-1}\rho) \big).
\end{cases}
\end{equation}
where $L$ is the linear operator $$L = -\nu \Delta + \lambda Id,$$ the closed convex set $\tilde{K}$ is defined by $$\tilde{K} := \{ g \in L^2(\mathbb{T}^d), L^{-1} g \leq 0\}$$ and $\bold{P}_A$ stands for the orthogonal projection in $L^2(\mathbb{T}^d)$ onto the set $A$. Let us note that from classical results on variational inequalities (see \citep{lions1967variational} for instance), $(u_n)_{n \geq 0}$ is a well defined sequence of $H^2(\mathbb{T}^d)$ because for all $n \in \mathbb{N}$, $m_n - L^{-1}\rho \in L^2(\mathbb{T}^d)$. Recalling the results of the previous section, $(m_n)_{n \geq 0}$ is well defined under some monotonicity assumptions on $f$. We have the following result :
\begin{Theorem}\label{convstop}
Assume that $f$ is $\alpha$ monotone from $L^2(\mathbb{T}^d)$ into itself for some $\alpha > 0$ and that $\delta < 2 \alpha$, then for any $u_0 \in H^2(\mathbb{T}^d)$, the sequence $(m_n)_{n \geq 0}$ defined by (\ref{iterationstop}) converges toward $m$ in $L^2(\mathbb{T}^d)$, where $(u,m)$ is the only solution of (\ref{optstopvi}).
\end{Theorem}
\begin{proof}
This result is a direct application of theorem \ref{newiteration}.
\end{proof}
Let us remark that the projection involved in (\ref{iterationstop}) is similar to the resolution of a bi-laplacian obstacle problem. Indeed, given $u_n, m_n$, we are looking for $u_{n+1}$ such that :
\[
\begin{aligned}
&\forall v \in H^2(\mathbb{T}^d), v \leq 0 :\\
& \int_{\mathbb{T}^d} (Lu_{n+1} - Lv)(Lu_{n+1} + \delta (m_n - L^{-1}\rho) - Lu_n) \leq 0.
\end{aligned}
\]
Let us also make a remark on the potential case. The potential case is the case when there exists $\mathcal{F} : L^2(\mathbb{T}^d) \to \mathbb{R}$ such that for every $m,m' \in L^2(\mathbb{T}^d)$ :
\begin{equation}\label{potential}
\underset{t \to 0}{\lim} \frac{\mathcal{F}(m + tm') - \mathcal{F}(m)}{t} = \int_{\mathbb{T}^d}f(m)(m' - m).
\end{equation}
In such a situation, if $f$ is strictly monotone, following the result of \citep{bertucci2017optimal}, the unique solution $(u,m)$ of (\ref{optstopvi}) is also the saddle point of the lagrangian $\mathcal{L}$ defined on $\{ µ \in L^2(\mathbb{T}^d), µ \geq 0\} \times \{v \in H^2(\mathbb{T}^d), v \leq 0\}$ by
\[
\mathcal{L}(µ,v) = \mathcal{F}(µ) + \int_{\mathbb{T}^d} (-\nu \Delta v + \lambda v)m - \int_{\mathbb{T}^d}v \rho.
\]
The iterations (\ref{iterationstop}) are then the result of the classical Uzawa's algorithm on $\mathcal{L}$.

\subsection{The case of impulse control}
In this section we are interested in building approximations of solutions of the following system :
\begin{equation}\label{systemimp}
\begin{cases}
\forall v \in H^1(\mathbb{T}^d), v \leq Mv :\\
(f(m) + \nu \Delta u - \lambda u, v -u)_{H^{-1}\times H^1} \geq 0;\\
(-\nu \Delta m + \lambda m - \rho, v - u)_{H^{-1}\times H^1} \geq 0 ;\\
\int_{\mathbb{T}^d}(f(m) + \nu \Delta u - \lambda u)m = 0;\\
u \leq Mu ; m \geq 0;
\end{cases}
\end{equation}
where $f$ is a continuous application from $L^2(\mathbb{T}^d)$ into itself, bounded uniformly from below on the positive elements of $L^2(\mathbb{T}^d)$, $\nu, \lambda > 0$ are two parameters of the model, $\rho \in H^{-1}(\mathbb{T}^d)$ is the entry rate of players and $M$ is the operator defined by $$Mv(x) := \underset{\xi \in J}{\inf} \{k(x, \xi) + v(x + \xi)\}$$ where $J$ is a finite set of $\mathbb{T}^d$ and $k$ is a smooth non-negative function.

The system (\ref{systemimp}) models Nash equilibria of MFG of impulse control in which the players face the running cost $f(m)$ and have to pay $k(x,\xi)$ if they are in $x$ to jump $\xi$ further. The density of players is $m$ and $u$ represents the value function of a generic player. We refer to \citep{bertucci2018fokker} for more details on this problem and for the following result. If $f$ is strictly monotone and $k$ satisfies $$\begin{cases} x \to \underset{\xi \in J}{\inf} \{k(x, \xi)\} \in W^{2, \infty};\\\exists k_0>0, \forall x \in \mathbb{T}^d, \xi \in J: k(x,\xi) \geq k_0 \end{cases}$$ then there exists a unique solution $(u,m) \in H^2(\mathbb{T}^d) \times H^1(\mathbb{T}^d)$ of (\ref{systemimp}), moreover, this couple $(u,m)$ satisfies :
\begin{equation}\label{impgoodvi}
\begin{cases}
 \forall µ \in L^2(\mathbb{T}^d), µ \geq 0:\\
\int_{\mathbb{T}^d} (f(m) + \nu \Delta u - \lambda u) (µ - m) \geq 0;\\
\forall v \in H^1(\mathbb{T}^d), v \leq Mv :\\
(-\nu \Delta m + \lambda m - \rho, v - u)_{H^{-1} \times H^1} \geq 0.
\end{cases}
\end{equation}
Thus we define, as in the case of optimal stopping, the following Uzawa's iterations for $\delta > 0$ :
\begin{equation}\label{iterationimp}
\begin{cases}
u_0 \in H^2(\mathbb{T}^d), u_0 \leq Mu_0.\\
m_n \in L^2(\mathbb{T}^d) \text{ defined by : } \\
m_n \geq 0; \forall µ \in L^2(\mathbb{T}^d), µ \geq 0 : \int_{\mathbb{T}^d}(f(m_n) + \nu \Delta u_n - \lambda u_n)(µ - m_n) \geq 0.\\
u_{n+1} \in H^2(\mathbb{T}^d) \text{ defined by : } L u_{n+1}  = \bold{P}_{K'} \big( L u_{n} - \delta(m_n - L^{-1}\rho) \big);
\end{cases}
\end{equation}
where $L$ still denotes the linear operator $$L = -\nu \Delta + \lambda Id,$$ the closed convex set $K'$ is defined by $$K' := \{ g \in L^2(\mathbb{T}^d), L^{-1}g \leq M (L^{-1}g)\}$$ and $\bold{P}_A$ stands for the orthogonal projection onto $A$ in $L^2(\mathbb{T}^d)$. We have the following result of convergence :
\begin{Theorem}\label{convimp}
Assume that $f$ is $\alpha$ monotone from $L^2(\mathbb{T}^d)$ into itself for some $\alpha > 0$ and that $\delta < 2 \alpha$, then for any $u_0 \in H^2(\mathbb{T}^d)$, the sequence $(m_n)_{n \geq 0}$ defined by (\ref{iterationstop}) converges toward $m$ in $L^2(\mathbb{T}^d)$, where $(u,m)$ is the only solution of (\ref{optstopvi}).
\end{Theorem}
\begin{proof}
This result is once again a direct application of theorem \ref{newiteration}.
\end{proof}
Let us remark that although the Hamilton-Jacobi-Bellmann equation in (\ref{systemimp}) is a quasi-variational inequality, the equation we have to solve at each iteration in (\ref{iterationimp}) to update the lagrange multiplier $u_n$ is a variational inequality, which is in principle easier to solve than a quasi-variational inequality.\\

Moreover, in the potential case, i. e. when there exists $\mathcal{F}$ satisfying (\ref{potential}), if $f$ is strictly monotone, the solution $(u,m)$ of (\ref{systemimp}) is the saddle point of $\mathcal{L}$ defined on $\{m \in L^2(\mathbb{T}^d), m \geq 0\} \times \{u \in H^2(\mathbb{T}^d), u \leq Mu\}$ by :
\[
\mathcal{L}(µ,v) = \mathcal{F}(µ) + \int_{\mathbb{T}^d} (-\nu \Delta v + \lambda v)m + \int_{\mathbb{T}^d}v \rho.
\]
This results can be found in \citep{bertucci2018fokker}. The iterations (\ref{iterationimp}) are then the ones from the use of the classical Uzawa's algorithm on $\mathcal{L}$.

\subsection{The case of continuous control}
We end this list of applications of Uzawa's iterations with the construction of approximating sequences for the following MFG system :
\begin{equation}\label{systemMFG}
\begin{cases}
-\nu \Delta u + \lambda u + H(x, \nabla u) = f(m) \text{ in } \mathbb{T}^d;\\
-\nu \Delta m + \lambda m - \text{div}(D_pH(x, \nabla u)m) = \rho \text{ in } \mathbb{T}^d;
\end{cases}
\end{equation}
where $f$ is the running cost of the players and the hamiltonian $H(x,p)$ is assumed to be convex in its second variable and uniformly lispchitz. We refer the reader to \citep{lasry2007mean, lions2007cours} for a full presentation and results on the system (\ref{systemMFG}). If $(u,m) \in H^2(\mathbb{T}^d) \times L^2(\mathbb{T}^d)$ is a solution of (\ref{systemMFG}) (with $m$ being a weak solution of the Fokker-Planck equation), then it is also a solution of :
\begin{equation}\label{vicontinuous}
\begin{cases}
\forall µ \in L^2(\mathbb{T}^d), µ \geq 0 :\\
\int_{\mathbb{T}^d} (f(m) + \nu \Delta u - \lambda u - H(x, \nabla u))(µ - m) \geq 0;\\
\forall v \in H^2(\mathbb{T}^d) :\\
\int_{\mathbb{T}^d} (- \nu \Delta (v - u) + \lambda (v -u) + D_pH(x, \nabla u)\cdot \nabla (v -u))m - \int_{\mathbb{T}^d} \rho (v- u) \geq 0.
\end{cases}
\end{equation}
Under the assumption that $f$ is strictly monotone, there exists at most one solution $(u,m)\in H^2(\mathbb{T}^d)\times L^2(\mathbb{T}^d)$ of (\ref{vicontinuous}). Although this system does not allow a direct application of theorem \ref{newiteration}, the convexity of the hamiltonian allows us to prove a result of convergence for Uzawa's like iterations. Given a sequence of non-negative real numbers $(\delta_n)_{n \geq 0}$, we define Uzawa's iteration in this case by :
\begin{equation}\label{iterationcontinuous}
\begin{cases}
u_0 \in H^2(\mathbb{T}^d);\\
m_n \text{ is defined by } \forall µ \in L^2(\mathbb{T}^d), µ \geq 0 :\\
\int_{\mathbb{T}^d}(f(m_n) + \nu \Delta u_n - \lambda u_n - H(x, \nabla u_n))(µ - m_n) \geq 0;\\
u_{n+1} \text{ is defined by } :\\
-\nu \Delta u_{n+1} + \lambda u_{n +1} + H(x, \nabla u_{n+1}) = -\nu \Delta u_{n} + \lambda u_{n} + H(x, \nabla u_{n}) - \delta_n\big(m_n - L_{u_n}^{*-1} (\rho)\big);
\end{cases}
\end{equation}
where for all $v \in H^1(\mathbb{T}^d)$, $L_v$ is the operator defined by :$$L_v w = -\nu \Delta w + \lambda w - D_pH(x, \nabla v)\cdot \nabla w.$$ We now establish the following result :
\begin{Theorem}\label{convcontinuous}
Assume that there exists a solution $(u,m) \in H^2(\mathbb{T}^d)\times L^2(\mathbb{T}^d)$ of (\ref{systemMFG}) and that $f$ is $\alpha$ monotone. Then there exists a sequence of non-negative real number $(\delta_n)_{n \geq 0}$ such that the iterations $(u_n, m_n)_{n \geq 0}$ defined by (\ref{iterationcontinuous}) are such that $(m_n)_{n \geq 0}$ converges toward $m$ in $L^2(\mathbb{T}^d)$.
\end{Theorem}
\begin{proof}
We denote by $(u,m)\in H^2(\mathbb{T}^d)\times L^2(\mathbb{T}^d)$ the unique solution of (\ref{vicontinuous}). We take a sequence $(\delta_n)_{n \geq 0}$, $\delta_n > 0$ for all $n \geq 0$ and we consider the iterations $(u_n,m_n)_{n \geq 0}$ given by (\ref{iterationcontinuous}) for a fixed $u_0\in H^2(\mathbb{T}^d)$. We introduce the notation $$HJB(v) := -\nu \Delta v + \lambda v + H(x, \nabla v).$$ Let us remark that for all $n \geq 0$:
\[
\begin{aligned}
HJB(u_{n+1}) - HJB(u) = HJB(u_{n}) - HJB(u) - \delta_n(m_n  -L_{u_n}^{*-1} \rho).
 \end{aligned}
\] 
Thus we obtain that :
\begin{equation}\label{proofcontinuous}
\begin{aligned}
||HJB(u_{n+1}) - HJB(u)||_{L^2}^2 = &||HJB(u_{n}) - HJB(u)||_{L^2}^2 + \delta_n^2 ||m_n - L_{u_n}^{*-1}\rho||^2_{L^2}\\
 &- 2 \delta_n \int_{\mathbb{T}^d} (HJB(u_{n}) - HJB(u))(m_n  -L_{u_n}^{*-1} \rho).
\end{aligned}
\end{equation}
We now make some calculations around the third term of the right hand side of the previous equality. We compute :
\[
\begin{aligned}
\int_{\mathbb{T}^d} (HJB(u_{n}) - HJB(u))(m_n  -L_{u_n}^{*-1} \rho) =&  \int_{\mathbb{T}^d} (HJB(u_n) - HJB(u))(m_n  -m) \\
 &+ \int_{\mathbb{T}^d} \big( HJB( u_n) - HJB( u) \big)(m-L_{u_n}^{*-1}\rho ).
 \end{aligned}
 \]
Because $m$ is the solution of $$L_u^* m = \rho,$$ we deduce from the convexity of $H$ :
\[
\int_{\mathbb{T}^d} \big( HJB( u_n) - HJB( u) \big)m \geq \int_{\mathbb{T}^d} \rho (u_n - u).
\]
On the other hand :
\[
\begin{aligned}
\int_{\mathbb{T}^d} \big( HJB( u) - HJB( u_n) \big)L_{u_n}^{*-1}\rho =  & \int_{\mathbb{T}^d} L_{u_n}(u - u_n)(L_{u_n}^{*-1} \rho)\\
& - \int_{\mathbb{T}^d} D_pH(x, \nabla u_n)\cdot \nabla( u -u_n)( L_{u_n}^{* -1} \rho)\\
& + \int_{\mathbb{T}^d} \big( H(x, \nabla u) - H(x, \nabla u_n) \big) (L_{u_n}^{*-1}\rho)\\
\end{aligned}
\]
By the maximum principle, $L_{u_n}^{*-1}\rho \geq 0$, thus we deduce from the convexity of the hamiltonian that :
\[
\int_{\mathbb{T}^d} \big( HJB( u) - HJB( u_n) \big)L_{u_n}^{*-1}\rho \geq \int_{\mathbb{T}^d} \rho (u - u_n).
\]
This inequality, together with the previous one implies that :
\[
\int_{\mathbb{T}^d} (HJB(u_{n}) - HJB(u))(m_n  -L_{u_n}^{-1} \rho) \geq  \int_{\mathbb{T}^d} (HJB(u_n) - HJB(u))(m_n  -m).
\]
Using the equation satisfied by $u$ and the definition of $m_n$, we obtain that :
\[
\int_{\mathbb{T}^d} (HJB(u_n) - HJB(u))(m_n  -m) \geq \int_{\mathbb{T}^d}\big( f(m_n) - f(m) \big) (m_n - m).
\]
The $\alpha$ convexity of $f$ yields finally :
\[
\int_{\mathbb{T}^d} (HJB(u_{n}) - HJB(u))(m_n  -L_{u_n}^{-1} \rho) \geq \alpha ||m_n - m||_{L^2}^2.
\]
Using this inequality in (\ref{proofcontinuous}) we obtain that:
\[
\begin{aligned}
\delta_n\big( 2 \alpha||m_n - m||_{L^2}^2 - \delta_n ||m_n - L_{u_n}^{*-1} \rho||_{L^2}^2 \big) \leq & ||HJB(u_{n}) - HJB(u)||_{L^2}^2\\
& - ||HJB(u_{n+1}) - HJB(u)||_{L^2}^2.
\end{aligned}
\]
We assume in a first time that for all $n \in \mathbb{N}$ :
\begin{equation}\label{hypproof}
\begin{cases}
||m_n - m||_{L^2} > 0;\\
||m_n - L_{u_n}^{*-1} \rho||_{L^2} > 0.
\end{cases}
\end{equation}
Then we define for all $n \in \mathbb{N}$ $$\delta_n = \frac{\alpha ||m_n - m||_{L^2}^2}{||m_n - L_{u_n}^{*-1}\rho||_{L^2}^2}.$$
Let us observe that in this situation the sequence $(||HJB(u_n) - HJB(u)||_{L^2})_{n \geq 0}$ is decreasing and thus it has a limit and $(u_n)_{n \geq 0}$ is a bounded sequence of $H^2(\mathbb{T}^d)$. We also remark that we deduce from the convergence of $(||HJB(u_n) - HJB(u)||_{L^2})_{n \geq 0}$ that 
\[
\frac{\alpha ||m_n - m||_{L^2}^2}{||m_n - L_{u_n}^{*-1}\rho||_{L^2}} \underset{n \to 0}{\longrightarrow} 0.
\]
Because $(||HJB(u_n)||_{L^2})_{n \geq 0}$ is bounded and $f$ is $\alpha$ monotone, we deduce that $(m_n)_{n \geq 0}$ is bounded in $L^2$ and thus that :
\[
||m_n - m||_{L^2} \underset{n \to 0}{\longrightarrow} 0.
\]
To complete the proof of the theorem, let us remark that if (\ref{hypproof}) is not satisfied for $n^* \in \mathbb{N}$, then $m_{n^*} = m$ and the convergence is also proved.
\end{proof}
\begin{Rem}
Let us remark that because of the $\alpha$ monotonicity of $f$, there are obvious estimates in $L^2(\mathbb{T}^d)$ for $(m_n)_{n \geq 0}$, thus the sequence $(\delta_n)_{n \geq 0}$ can be chosen to be an explicit constant.
\end{Rem}
In the potential case, when there exists $\mathcal{F}$ satisfying (\ref{potential}), and when $f$ is strictly monotone, the solution $(u,m)$ of (\ref{systemMFG}) is also the saddle point of the lagrangian $\mathcal{L}$ defined on $\{µ \in L^2(\mathbb{T}^d), µ \geq 0\} \times H^2(\mathbb{T}^d)$ by :
\[
\mathcal{L}(µ,v) = \mathcal{F}(µ) + \int_{\mathbb{T}^d} (\nu \Delta v - \lambda v - H(x, \nabla v))m - \int_{\mathbb{T}^d} \rho v.
\]
Even though the optimal control interpretation presented in \citep{lasry2007mean} is not exactly written in this form, it can be easily checked that the formulations are equivalent, at least formally. The iterations (\ref{iterationcontinuous}) are in this case the result of Uzawa's algorithm on $\mathcal{L}$, in the sense that they are formally the result of a gradient ascent method on the dual problem :
\[
\underset{v \in H^2(\mathbb{T}^d)}{\sup} \underset{ µ\geq 0, \in L^2(\mathbb{T}^d)}{\inf} \mathcal{L}(µ,v).
\]

\begin{Rem}
In the three cases mentioned above (optimal stopping, impulse control and continuous control), the sequence $(u_n)_{n \geq 0}$ defined by the Uzawa's iterations is always a bounded sequence of $H^2(\mathbb{T}^d)$. Therefore, up to a subsequence, $(u_n)_{n \geq 0}$ converges in $H^1(\mathbb{T}^d)$ toward $u \in H^2(\mathbb{T}^d)$. This function $u$ is in fact such that $(u,m)$ is the solution of the MFG system and thus the whole sequence $(u_n)_{n \geq 0}$ converges toward $u$.
\end{Rem}
\subsection{Other possible applications of Uzawa's iterations}
We give here some immediate applications of Uzawa's iterations. First let us note that the operator $-\nu \Delta + \lambda Id$ involved in the three problems above can be replaced by more general elliptic linear operators. Let us also mention that this method is also valid in more general domains than $\mathbb{T}^d$. This method can also be applied in time dependent situations.\\

Another important remark is that Uzawa's iterations can also be applied in the optimal control of PDE governed by inequalities, such that 
\[
\underset{m,  A(m) \leq 0}{\inf} F(m),
\]
where $A$ is a partial differential operator. Such a class of problem is of some importance. For instance we refer to \citep{bertucci2017optimal} for a heuristic argument on why (\ref{optstopvi}) can be interpreted as the optimality conditions for the relaxation of an optimal shape problem. The relaxation is then of the form just mentioned above.

\section{Numerical framework and numerical results}
We present here the discrete versions of the three problems mentioned in the previous section (optimal stopping, impulse control and continuous control MFG systems). We also present numerical results of the implementation of Uzawa's iteration for those problems. 
\subsection{Notations and presentation of the problem}
We give here the notations we are going to use to present the discretized problem we are interested in. We fix a non-negative integer $d$ and we define $h > 0$ by $h = d^{-1}$. We work here on a grid $G_d = \{ (i,j), 1 \leq i,j \leq d\}$ which we interpret as a discretization of the $2$ dimensional torus. Let $f_d : \mathbb{R}^{d^2} \to \mathbb{R}^{d^2}$ be a continuous application. We fix $\xi \in G_d$ and $k_0 > 0$ a real number. We then define for all $v \in \mathbb{R}^{d^2}$ $Mv$ by :
\[
(Mv)_{i,j} = k_0 + v_{(i,j) + \xi}.
\]
We denote by $g : (p_1,p_2,p_3,p_4) \to g(p_1,p_2,p_3,p_4)$ a discretization of the hamiltonian $H : \mathbb{R}^2 \to \mathbb{R}$ defined by $H(p) = \sqrt{1 + |p|^2}$. Thus $g$ is such that for $p_1, p_2 \in \mathbb{R}$ :
\[
g(p_1,p_1,p_2,p_2) = \sqrt{1 + (p_1)^2 + (p_2)^2};
\]
and $g$ is non decreasing with respect to $p_1$ and $p_3$ and non decreasing with respect to $p_2$ and $p_4$. We denote by $\nabla_p g$ the gradient of $g$. We also define the vector of derivatives $D_hv$ of a vector $v \in \mathbb{R}^{d^2}$ by :
\[
(D_hv)_{i,j} = \big( \frac{v_{i+1,j} - v_{i,j}}{h}, \frac{v_{i,j} - v_{i-1,j}}{h}, \frac{v_{i,j+1} - v_{i,j}}{h}, \frac{v_{i,j} - v_{i,j-1}}{h} \big).
\]
For $\nu, \lambda > 0$, we also define the discrete operator $A : \mathbb{R}^{d^2} \to \mathbb{R}^{d^2}$ by :
\[
(Av)_{i,j} = \nu \frac{4 v_{i,j} - v_{i+1,j} - v_{i-1,j} - v_{i,j+1} - v_{i,j-1}}{h} + \lambda v_{i,j}, \forall (i,j) \in G_d;
\]
where we use periodic boundary condition on $G_d$.\\
For $v \in \mathbb{R}^p$ for some $p\geq 0$, we use the notation $v \geq 0$ when for all $1\leq i \leq p$, $v_i \geq 0$. We take an element $\rho_d \in \mathbb{R}^d$, $\rho_d \geq 0$.\\
For the rest of this section, $\mathbb{R}^{d^2}$ is endowed with the scalar product :
\[
<x,y> = \sum_{1 \leq i,j\leq d} h^2 x_{i,j}y_{i,j}
\]

In this section, we present the results of the implementation of Uzawa's iterations to approximate the solutions of the following three problems (each time the unknown is the couple $(u,m)$) :
\begin{equation}\label{numstop}
\begin{cases}
<f_d(m) -Au , v - u> \leq 0 , \forall  v \in \mathbb{R}^{d^2}, v \leq 0;\\
<Au - Av, m> + < u - v, \rho_d> \leq 0 ; \forall v \in \mathbb{R}^{d^2}, v \leq 0;\\
< Au - f_d(m), m> = 0;\\
u \leq 0 ; m \geq 0.
\end{cases}
\end{equation}
\begin{equation}\label{numjump}
\begin{cases}
<f_d(m) -Au , v - u> \leq 0 , \forall  v \in \mathbb{R}^{d^2}, v \leq Mv;\\
<Au - Av, m> + < u - v, \rho_d> \leq 0 ; \forall v \in \mathbb{R}^{d^2}, v \leq Mv;\\
< Au - f_d(m), m> = 0;\\
u \leq Mu ; m \geq 0.
\end{cases}
\end{equation}
\begin{equation}\label{numcontinuous}
\begin{cases}
Au + g(Du) = f_d(m);\\
< Av + \nabla_p g(D_hu)\cdot D_h(v), m> - <v, \rho_d>  = 0 ; \forall v \in \mathbb{R}^{d^2};\\
m \geq 0.
\end{cases}
\end{equation}
Those problems are the discretized version of respectively (\ref{optstopvi}), (\ref{systemimp}) and (\ref{systemMFG}).

\subsection{A remark on the convergence of the discretized problems toward the continuous ones}
Although the convergence of (\ref{numstop}), (\ref{numjump}) and (\ref{numcontinuous}) toward their continuous version is not the objective of this article, we explain here briefly why such a convergence is expected. We give some results on the case of (\ref{numstop}). We refer to \citep{achdou2010mean} for results on (\ref{numcontinuous}).\\

We begin by detailing in which sense $(f_d)_{d \geq 1}$ converges toward $f : L^2(\mathbb{T}^2) \to L^2(\mathbb{T}^2)$. For any sequence $(m_d)_{d \geq 1}$, we define $(\tilde{m}_d)_{d \geq 1} \in (L^2(\mathbb{T}^2))^{\mathbb{N}}$ by
\[
\tilde{m}_d(x,y) = (m_d)_{i,j} \text{ if }\begin{cases} i-1 \leq x \times d < i ,\\
 j-1 \leq y \times d < j.
\end{cases}
\]
We assume that if $(\tilde{m}_d)_{d \geq 1}$ converges toward $m$ in $L^2(\mathbb{T}^2)$, then $(\tilde{f}_d(m_d))_{d \geq 1}$ converges toward $f(m)$ in $L^2(\mathbb{T}^d)$. We also assume, using the same notations, that $(\tilde{\rho}_d)_{d \geq 1}$ converges toward $\rho$ in $L^2(\mathbb{T}^2)$. We now start by proving a lemma which gives the main idea for the convergence of the finite problem.
\begin{Lemma}\label{lemmanum}
Let us assume that $f_d : \mathbb{R}^{d^2} \to \mathbb{R}^{d^2}$ is $\alpha$ monotone and that $(u,m)$ is the only solution of (\ref{numstop}). For any $\epsilon_1,\epsilon_2 > 0$ and $(v,µ)$ such that 
\begin{equation}\label{ineqlemme}
\begin{cases}
<f_d(µ) -Av , µ' - µ> \geq -\epsilon_1 , \forall  µ' \in \mathbb{R}^{d^2}, µ' \geq 0;\\
<Av - Av', µ> + < v - v', \rho> \leq \epsilon_2 ; \forall v' \in \mathbb{R}^{d^2}, v' \leq 0;\\
v \leq 0 ; µ \geq 0;
\end{cases}
\end{equation}
the following holds :
\[
||m - µ||^2 \leq \frac{\epsilon_1 + \epsilon_2}{\alpha}.
\]
\end{Lemma}
\begin{proof}
Because of the inequalities verified by $(v, µ)$,
\[
\begin{cases}
<f_d(µ) -Av , m - µ> \geq -\epsilon_1;\\
<Av - Au, µ> + < v - u, \rho> \leq \epsilon_2.
\end{cases}
\]
Therefore, using the fact that $(u,m)$ is the solution of (\ref{numstop}), we deduce that
\[
\begin{cases}
<f_d(µ) - f_d(m),m-µ> + <Au -Av , m - µ> \geq -\epsilon_1;\\
<Av - Au, µ - m> \leq \epsilon_2.
\end{cases}
\]
Thus we obtain that
\[
<f_d(µ) - f_d(m),µ- m> \leq \epsilon_1 + \epsilon_2.
\]
Using the $\alpha$ monotonicity of $f_d$ the result is proved.
\end{proof}
\begin{Rem}
If (\ref{ineqlemme}) is not satisfied for all $µ', v'$ but only for $m,u$, the results of the lemma still holds.
\end{Rem}
We now show an exemple of result of convergence.
\begin{Prop}
Let us assume that $f : L^2(\mathbb{T}^2) \to L^2(\mathbb{T}^2)$ is $\alpha$ monotone and let us assume that there exists $(u^*, m^*)$, unique solution of (\ref{optstopvi}). We also assume that for every $d\in \mathbb{N}, d \geq 1$, $f_d : \mathbb{R}^{d^2} \to \mathbb{R}^{d^2}$ is $\alpha$ monotone and that there exists a unique solution $(u_d,m_d)$ of (\ref{numstop}). Then the following holds 
\[
||m_d - m^*_d||_{\mathbb{R}^{d^2}} \underset{d \to \infty}{\longrightarrow} 0;
\]
where $m^*_d \in \mathbb{R}^{d^2}$ is such that defining $\tilde{m^*}$ by 
\[
\tilde{m^*}(x,y) = (m^*_d)_{i,j} \text{ if }\begin{cases} i-1 \leq x \times d < i ,\\
 j-1 \leq y \times d < j;
\end{cases}
\]
we have the convergence :
\[
||m^* - \tilde{m^*}||_{L^2(\mathbb{T}^2)} \underset{d \to \infty}{\longrightarrow} 0.
\]
\end{Prop}
\begin{proof}
We denote for all $d \geq 1 $ $(u_d,m_d)$ the only solution of (\ref{numstop}) and by $(u,m)$ the only solution of (\ref{optstopvi}). Our aim is to build for all $d\geq 1$, $v_d, µ_d \in \mathbb{R}^{d^2}$, an approximate solution of (\ref{numstop}) using $(u,m)$. We define $x_{i,j} = ((i-1)h, (j-1)h)$ for $1 \leq i,j \leq d$. And we consider a $C^{\infty}$ partition of the unity $(\phi_d^{i,j})_{1\leq i,j\leq d}$ subordinate to the cover $(B(x_{i,j}, h\sqrt{2}))_{1 \leq i,j \leq d}$, where $B(x,\delta)$ denotes the open ball of center $x$ and radius $\delta$. We define $(\varphi_d^{i,j})_{1\leq i,j\leq d}$ by $\varphi_d^{i,j} = (\int \phi_d^{i,j})^{-1}\phi_d^{i,j}$. We then define $v_d, µ_d \in \mathbb{R}^{d^2}$ by :
\[
(v_d)_{i,j} = (u*\varphi_d^{i,j})(x_{i,j})
\]
\[
(µ_d)_{i,j} = (m*\varphi_d^{i,j})(x_{i,j})
\]
It is easy to verify that there exists $(\epsilon_n)_{n \geq 0} \in \mathbb{R}^{\mathbb{N}}$ and $(\epsilon'_{n})_{n\geq 0}, (\tilde{\epsilon}_n)_{n \geq 0}$ such that $\epsilon'_{n}, \tilde{\epsilon}_n \in \mathbb{R}^{n^2}$ for every $n \geq 1$ and the three sequences converge to zero together with 
\[
< f_d(µ_d) - Av_d, µ_d> \leq \epsilon_d
\]
\[
< f_d(µ_d) - Av_d, m_d> \geq -<\epsilon'_d,m_d>
\]
\[
< A v_d, µ_d> - <v_d, \rho_d> \leq \epsilon_d 
\]
\[
< A u_d, µ_d> - <u_d, \rho_d> \geq -<\tilde{\epsilon}_d,Au_d>
\]
Thus using lemma \ref{lemmanum} (and the remark following) we deduce that
\[
<f(µ_d) - f(m_d), µ_d - m_d> \leq  < \epsilon'_d, m_d> +<\tilde{\epsilon}_d , Au_d> + 2 \epsilon_d.
\]
Using estimates on $(Au_n)_{n \geq 0}$ and $(m_n)_{n \geq 0}$ (which are easy to obtain) we deduce, using the $\alpha$ monotonicity of $f_d$ that
\[
<m_d - µ_d, m_d - µ_d>  \underset{ d \to \infty}{\longrightarrow} 0.
\]
The result then follows.
\end{proof}

\subsection{Numerical results}
\subsubsection{The optimal stopping case}
In figure \ref{fig:stop}, we give the density $m$ and its Lagrange multiplier $u$ obtained after $20$ Uzawa's iterations. We use a standard Uzawa's algorithm to perform numerically the projection which updates the Lagrange multiplier at each step. The parameters of the model are 

\begin{center}
   \begin{tabular}{| l | c | }
     \hline
     $\nu$ & $0.02$  \\ \hline
     $\lambda$ & $1 $ \\ \hline
     $d$ & $40$  \\ \hline
     $f(m)$ & $f_0 + m + (- \Delta + Id)^{-1}m$  \\ \hline
     $f_0$ & $\cos(2\pi x) + 2 \cos(2 \pi (y-x)) + \cos(6 \pi x)$ \\ \hline
     $\delta$ & $0.5$\\ \hline
     $\rho$ & $1$\\
     \hline
   \end{tabular}
 \end{center}

\begin{figure}
\centering
\begin{subfigure}{.5\textwidth}
  \centering
  \includegraphics[width=.9\linewidth]{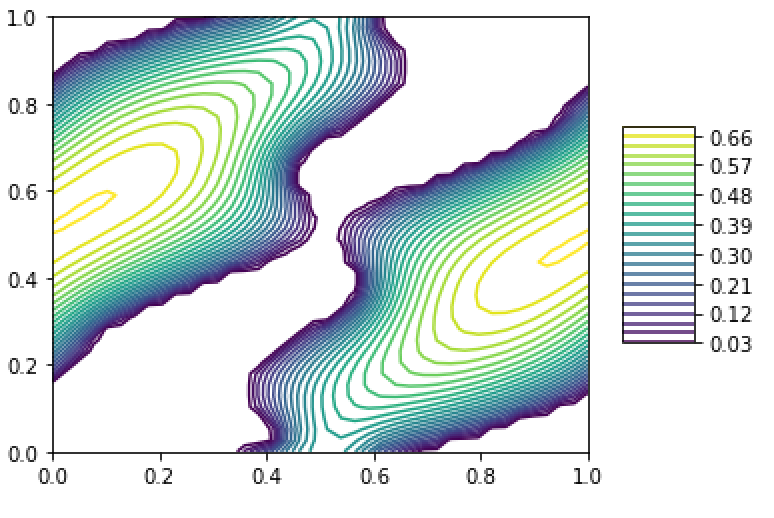}
  \caption{Contours of $m$}
  \label{fig:sub1}
\end{subfigure}%
\begin{subfigure}{.5\textwidth}
  \centering
  \includegraphics[width=.9\linewidth]{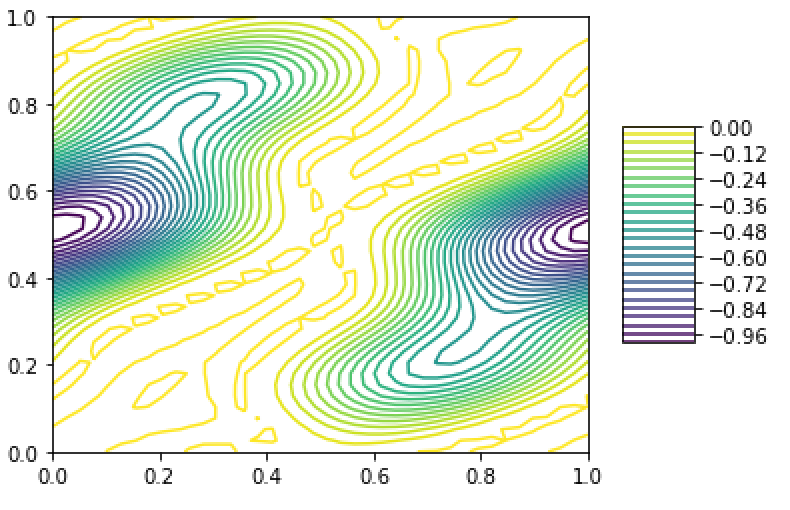}
  \caption{Contours of $u$}
  \label{fig:sub2}
\end{subfigure}
\begin{subfigure}{.5\textwidth}
  \centering
  \includegraphics[width=.9\linewidth]{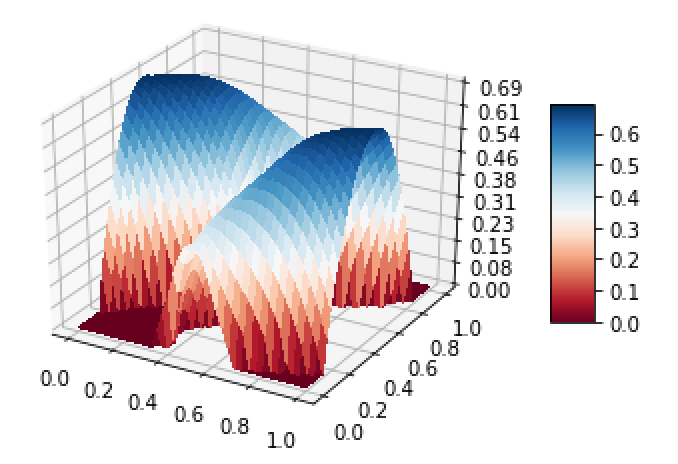}
  \caption{Graph of $m$}
  \label{fig:sub3}
\end{subfigure}%
\begin{subfigure}{.5\textwidth}
  \centering
  \includegraphics[width=.9\linewidth]{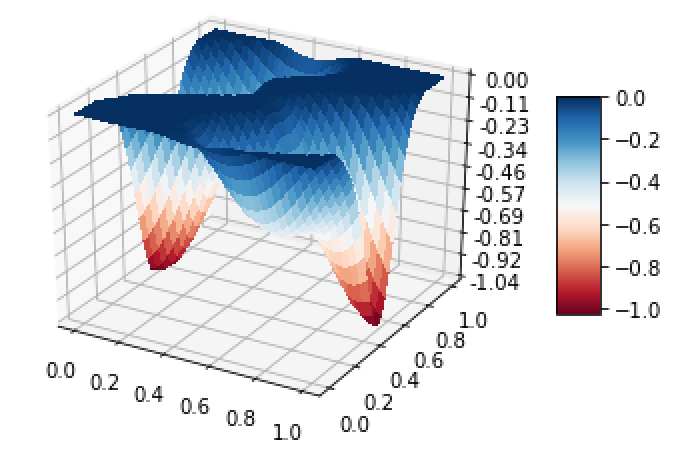}
  \caption{Graph of $u$}
  \label{fig:sub4}
\end{subfigure}
\caption{Uzawa's iterations for (\ref{numstop})}
\label{fig:stop}
\end{figure}

\subsubsection{The impulse control case}
In figure \ref{fig:imp} we give the density $m$ and its Lagrange multiplier $u$ obtained after $40$ Uzawa's iterations. We use at each step Uzawa's algorithm to perform numerically the projection which updates the Lagrange multiplier in our Uzawa's iterations. The parameters of the model are 

\begin{center}
   \begin{tabular}{| l | c | }
     \hline
     $\nu$ & $0.02$  \\ \hline
     $\lambda$ & $1 $ \\ \hline
     $d$ & $40$  \\ \hline
     $f(m)$ & $f_0 + m + (-\Delta + Id)^{-1}m$  \\ \hline
     $f_0$ & $\cos(2\pi x) + 2 \cos(2 \pi (y-x)) + \cos(6 \pi x)$ \\ \hline
     $k_0$ & $0.5$ \\ \hline
     $\xi$ & $ (20/7,0)$ \\ \hline
     $\delta$ & $0.5$\\ \hline
     $\rho$ & $1$\\
     \hline
   \end{tabular}
 \end{center}

\begin{figure}
\centering
\begin{subfigure}{.5\textwidth}
  \centering
  \includegraphics[width=.9\linewidth]{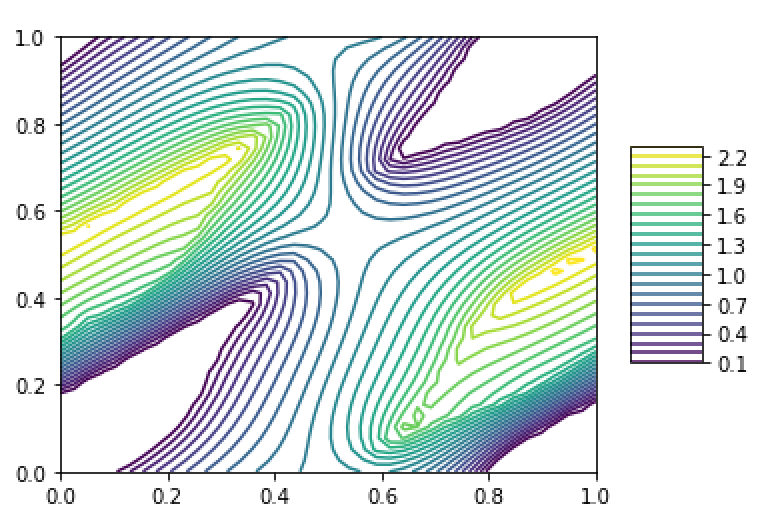}
  \caption{Contours of $m$}
  \label{fig:sub5}
\end{subfigure}%
\begin{subfigure}{.5\textwidth}
  \centering
  \includegraphics[width=.9\linewidth]{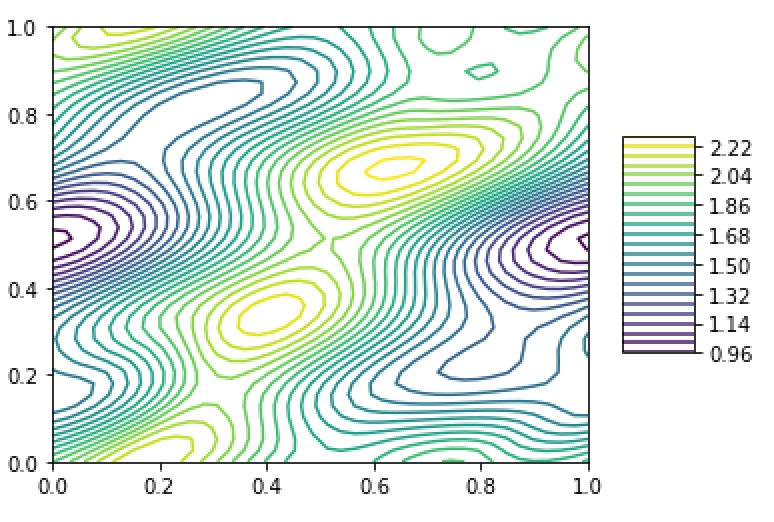}
  \caption{Contours of $u$}
  \label{fig:sub6}
\end{subfigure}
\begin{subfigure}{.5\textwidth}
  \centering
  \includegraphics[width=.9\linewidth]{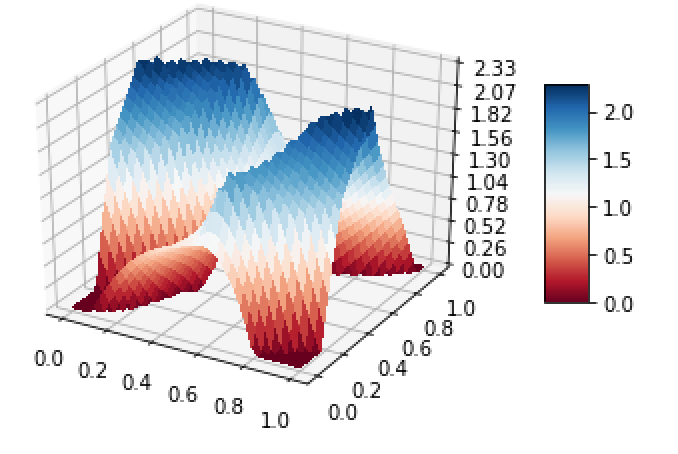}
  \caption{Graph of $m$}
  \label{fig:sub7}
\end{subfigure}%
\begin{subfigure}{.5\textwidth}
  \centering
  \includegraphics[width=.9\linewidth]{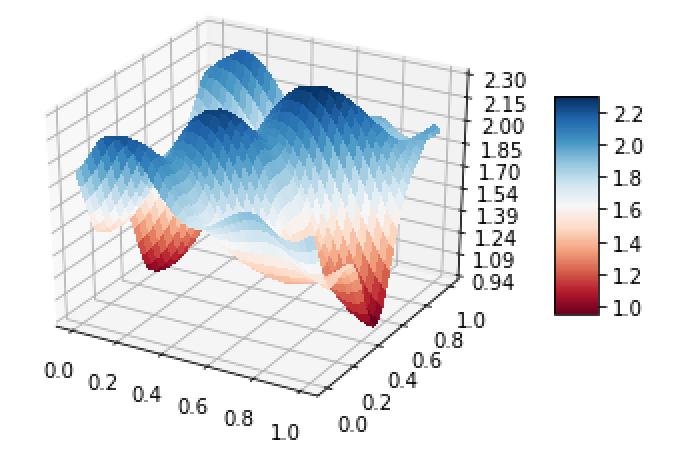}
  \caption{Graph of $u$}
  \label{fig:sub8}
\end{subfigure}
\caption{Uzawa's iterations for (\ref{numjump})}
\label{fig:imp}
\end{figure}

\subsubsection{The continuous control case}
In figure \ref{fig:cont} we give the density $m$ and its Lagrange multiplier $u$ obtained after $3000$ Uzawa's iterations. We use a standard Newton method on the finite differences scheme  at each step to solve the Hamilton-Jacobi-Bellmann equation which updates the Lagrange multiplier. To compute at each iteration $n$ the value of $L_{u_n}^{*-1} \rho$, we use a biconjugate gradient stabilized method. The parameters of the model are 

\begin{center}
   \begin{tabular}{| l | c | }
     \hline
     $\nu$ & 0.05  \\ \hline
     $\lambda$ & 1  \\ \hline
     d & 40  \\ \hline
     $f(m)$ & $f_0 + m + (-\Delta + Id)^{-1} m$ \\ \hline
     $f_0(x,y)$ & $\cos(2 \pi x) + \cos (2 \pi y) + \cos(4 \pi x)$ \\ \hline
     $H(p)$ & $\sqrt{1 + |p|^2}$ \\ \hline
     $(\delta_n)_{n \geq 0}$ & 0.05 $\forall n$ \\ \hline
     $\rho$ & 1\\
     \hline
   \end{tabular}
 \end{center}

\begin{figure}
\centering
\begin{subfigure}{.5\textwidth}
  \centering
  \includegraphics[width=.9\linewidth]{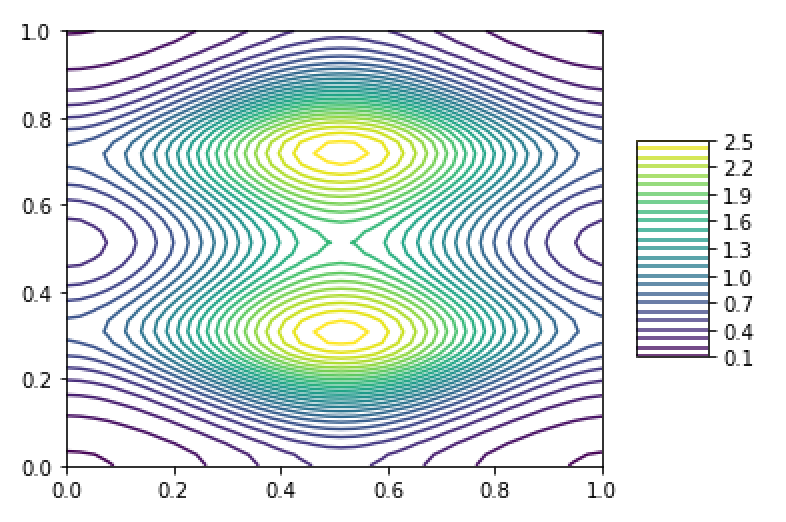}
  \caption{Contours of $m$}
  \label{fig:sub9}
\end{subfigure}%
\begin{subfigure}{.5\textwidth}
  \centering
  \includegraphics[width=.9\linewidth]{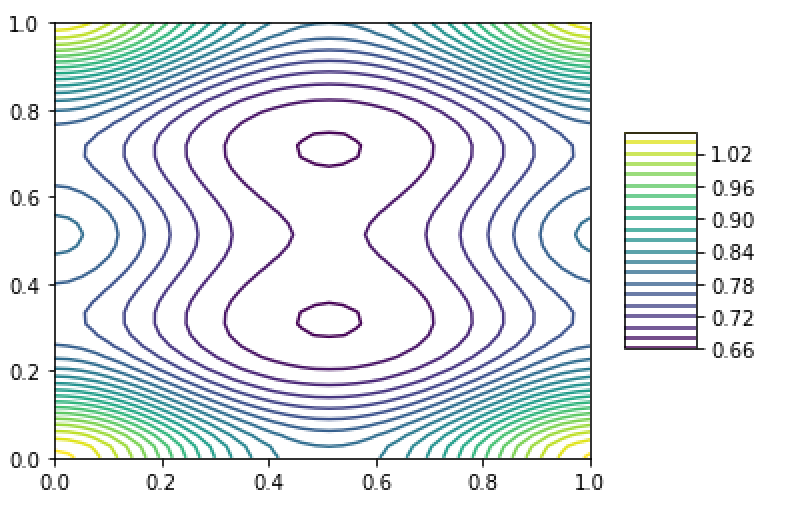}
  \caption{Contours of $u$}
  \label{fig:sub10}
\end{subfigure}
\begin{subfigure}{.5\textwidth}
  \centering
  \includegraphics[width=.9\linewidth]{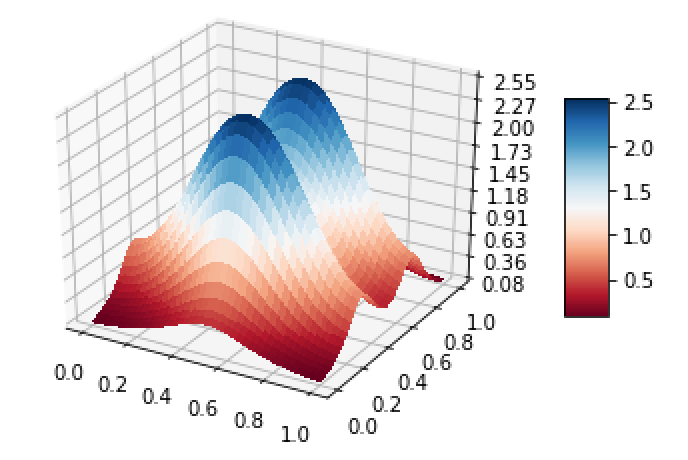}
  \caption{Graph of $m$}
  \label{fig:sub11}
\end{subfigure}%
\begin{subfigure}{.5\textwidth}
  \centering
  \includegraphics[width=.9\linewidth]{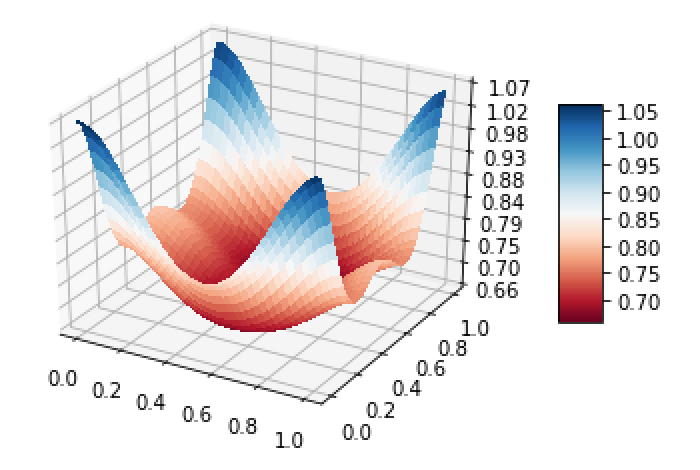}
  \caption{Graph of $u$}
  \label{fig:sub12}
\end{subfigure}
\caption{Uzawa's iterations for (\ref{numcontinuous})}
\label{fig:cont}
\end{figure}

\section*{Acknowledgments}
I would like to thank Pr. Yves Achdou (Universit\'e Paris Descartes) for his helpful advices on MFG and numerical simulations.

\bibliographystyle{plainnat}
\bibliography{bertucci_2018_uzawa}
\end{document}